\documentclass[10pt,a4paper,reqno]{amsart}
\usepackage[british]{babel}

\usepackage{amssymb}
\usepackage{amsmath}
\usepackage{amsthm}
\usepackage{array}
\usepackage{enumitem}
\usepackage{mathrsfs}
\usepackage[hidelinks]{hyperref}
\usepackage{mathrsfs}

\newtheorem{thm}{Theorem}[section]

\newtheorem{cor}[thm]{Corollary}
\newtheorem{lem}[thm]{Lemma}

\newtheorem{prop}[thm]{Proposition}

\theoremstyle{definition}

\newtheorem{exmpl}[thm]{Example}

\newtheorem{remark}[thm]{Remark}

\renewcommand{\epsilon}{\varepsilon}
\renewcommand{\phi}{\varphi}
\newcommand{\defeq}{\mathrel{\mathop:}=}                         
          
\DeclareMathOperator{\id}{id}
\DeclareMathOperator{\Sym}{Sym}
\DeclareMathOperator{\gr}{gr}
\DeclareMathOperator{\asdim}{asdim}

\makeatletter
\def\moverlay{\mathpalette\mov@rlay}
\def\mov@rlay#1#2{\leavevmode\vtop{%
		\baselineskip\z@skip \lineskiplimit-\maxdimen
		\ialign{\hfil$\m@th#1##$\hfil\cr#2\crcr}}}
\newcommand{\charfusion}[3][\mathord]{
	#1{\ifx#1\mathop\vphantom{#2}\fi
		\mathpalette\mov@rlay{#2\cr#3}
	}
	\ifx#1\mathop\expandafter\displaylimits\fi}
\makeatother

\newcommand{\cupdot}{\charfusion[\mathbin]{\cup}{\cdot}}
\newcommand{\bigcupdot}{\charfusion[\mathop]{\bigcup}{\cdot}}

\begin{document}
		
	\setlist{noitemsep}
	
\title[About von Neumann's problem for locally compact groups]{About von Neumann's problem for locally compact groups}
	
\author{Friedrich Martin Schneider}
\address{Institute of Algebra, TU Dresden, 01062 Dresden, Germany}
\curraddr{Department of Mathematics, The University of Auckland, Private Bag 92019, Auckland 1142, New Zealand}
\email{martin.schneider@tu-dresden.de}
\thanks{This research has been supported by funding of the German Research Foundation (reference no.~SCHN 1431/3-1) as well as by funding of the Excellence Initiative by the German Federal and State Governments.}
	
	
\date{\today}

\begin{abstract} 
	We note a generalization of Whyte's geometric solution to the von Neumann problem for locally compact groups in terms of Borel and clopen piecewise translations. This strengthens a result of Paterson on the existence of Borel paradoxical decompositions for non-amenable locally compact groups. Along the way, we study the connection between some geometric properties of coarse spaces and certain algebraic characteristics of their wobbling groups.
\end{abstract}

\maketitle
	
\section{Introduction}
	
In his seminal article~\cite{vonNeumann} von Neumann introduced the concept of amenability for groups in order to explain why the Banach-Tarski paradox occurs only for dimension greater than two. He proved that a group containing an isomorphic copy of the free group $F_{2}$ on two generators is not amenable. The converse, i.e., the question whether every non-amenable group would have a subgroup being isomorphic to $F_{2}$, was first posed in print by Day~\cite{day}, but became known as the \emph{von Neumann problem} (or sometimes \emph{von Neumann-Day problem}). The original question has been answered in the negative by Ol'{\v{s}}anski{\u\i}~\cite{olshanskii}. However, there are very interesting positive solutions to variants of the von Neumann problem in different settings: a geometric solution by Whyte~\cite{whyte}, a measure-theoretic solution by Gaboriau and Lyons~\cite{GaboriauLyons} and its generalization to locally compact groups by Gheysens and Monod~\cite{GheysensMonod}, as well as a Baire category solution by Marks and Unger~\cite{MarksUnger}. Whyte's geometric version reads as follows.
	
\begin{thm}[Theorem 6.2 in~\cite{whyte}]\label{theorem:whyte.1} A uniformly discrete metric space of uniformly bounded geometry is non-amenable if and only if it admits a partition whose pieces are uniformly Lipschitz embedded copies of the $4$-regular tree. \end{thm}
	
In particular, the above applies to Cayley graphs of finitely generated groups and in turn yields a geometric solution to the von Neumann problem.
	
The aim of the present note is to extend Whyte's relaxed version of the von Neumann conjecture to the realm of locally compact groups. For this purpose, we need to view the result from a slightly different perspective. Given a uniformly discrete metric space $X$, its \emph{wobbling group} (or \emph{group of bounded displacement}) is defined as \begin{displaymath}
	\mathscr{W}(X) \defeq \{ \alpha \in \Sym (X) \mid \exists r \geq 0 \, \forall x \in X \colon \, d(x,\alpha (x)) \leq r \} .
\end{displaymath} Wobbling groups have attracted growing attention in recent years~\cite{JuschenkoMonod,juschenko,cornulier}. Since the $4$-regular tree is isomorphic to the standard Cayley graph of $F_{2}$, one can easily reformulate Whyte's in terms of semi-regular subgroups. Let us recall that a subgroup $G \leq \Sym (X)$ is said to be \emph{semi-regular} if no non-identity element of $G$ has a fixed point in $X$.
	
\begin{cor}[Theorem 6.1 in~\cite{whyte}]\label{corollary:whyte.1} A uniformly discrete metric space $X$ of uniformly bounded geometry is non-amenable if and only if $F_{2}$ is isomorphic to a semi-regular subgroup of $\mathscr{W}(X)$. \end{cor}
	
For a finitely generated group $G$, the metrics generated by any two finite symmetric generating sets containing the neutral element are equivalent and hence give rise to the very same wobbling group~$\mathscr{W}(G)$. It is easy to see that $\mathscr{W}(G)$ is just the group of piecewise translations of $G$, i.e., a bijection $\alpha \colon G \to G$ belongs to $\mathscr{W}(G)$ if and only if exists a finite partition $\mathscr{P}$ of $G$ such that \begin{displaymath}
	\forall P \in \mathscr{P} \, \exists g \in G \colon \quad \alpha\vert_{P} = \lambda_{g}\vert_{P} .
\end{displaymath} Furthermore, we note that the semi-regularity requirement in the statement above cannot be dropped: in fact, van Douwen~\cite{VanDouwen} showed that $\mathscr{W}(\mathbb{Z})$ contains an isomorphic copy of $F_{2}$, despite $\mathbb{Z}$ being amenable. As it turns out, $F_{2}$ embeds into the wobbling group of any coarse space of positive asymptotic dimension (see Proposition~\ref{proposition:dimension} and Remark~\ref{remark:unbounded}).
	
We are going to present a natural counterpart of Corollary~\ref{corollary:whyte.1} for general locally compact groups. Let $G$ be a locally compact group. We call a bijection $\alpha \colon G \to G$ a \emph{clopen piecewise translation} of $G$ if there exists a finite partition $\mathscr{P}$ of $G$ into clopen subsets such that \begin{displaymath}
	\forall P \in \mathscr{P} \, \exists g \in G \colon \quad \alpha\vert_{P} = \lambda_{g}\vert_{P} ,
\end{displaymath} i.e., on every member of $\mathscr{P}$ the map $\alpha$ agrees with a left translation of $G$. It is easy to see that the set $\mathscr{C}(G)$ of all clopen piecewise translations of $G$ constitutes a subgroup of the homeomorphism group of the topological space $G$ and that the mapping $\Lambda \colon G \to \mathscr{C}(G), \, g \mapsto \lambda_{g}$ embeds $G$ into $\mathscr{C}(G)$ as a \emph{regular}, i.e., semi-regular and transitive, subgroup. Similarly, a bijection $\alpha \colon G \to G$ is called a \emph{Borel piecewise translation} of $G$ if there exists a finite partition $\mathscr{P}$ of $G$ into Borel subsets with \begin{displaymath}
	\forall P \in \mathscr{P} \, \exists g \in G \colon \quad \alpha\vert_{P} = \lambda_{g}\vert_{P} .
\end{displaymath} Likewise, the set $\mathscr{B}(G)$ of all Borel piecewise translations of $G$ is a subgroup of the automorphism group of the Borel space of $G$ and contains $\mathscr{C}(G)$ as a subgroup.

For a locally compact group $G$, both $\mathscr{B}(G)$ and $\mathscr{C}(G)$ are reasonable analogues of the wobbling group. Yet, the mere existence of an embedding of~$F_{2}$ as a semi-regular subgroup of $\mathscr{B}(G)$, or even $\mathscr{C}(G)$, does not prevent $G$ from being amenable. In fact, there are many examples of compact (thus amenable) groups that admit $F_{2}$ as a (non-discrete) subgroup and hence as a semi-regular subgroup of $\mathscr{C}(G)$. For example, since $F_{2}$ is residually finite, it embeds into the compact group formed by the product of its finite quotients. Therefore, we have to seek for a topological analogue of semi-regularity, which amounts to a short discussion.
	
\begin{remark}\label{remark:semiregular} Let $X$ be a set. A subgroup $G \leq \Sym (X)$ is semi-regular if and only if there exists a (necessarily surjective) map $\psi \colon X \to G$ such that $\psi (gx) = g\psi (x)$ for all $g \in G$ and $x \in X$. Obviously, the latter implies the former. To see the converse, let $\sigma \colon X \to X$ be any orbit cross-section for the action of $G$ on $X$, i.e., ${\sigma (X)} \cap {Gx} = \{ \sigma (x) \}$ for every $x \in X$. Since $G$ is semi-regular, for each $x \in X$ there is a unique $\psi (x) \in G$ such that $\psi (x)\sigma (x) = x$. For all $g \in G$ and $x \in X$, we have \begin{displaymath}
	g\psi (x)\sigma (gx) = g\psi (x)\sigma (x) = gx = \psi (gx) \sigma (gx),
\end{displaymath} which readily implies that $\psi (gx) = g\psi (x)$. So, $\psi \colon X \to G$ is as desired. \end{remark}	
	
The purpose of this note is to show the following.
	
\begin{thm}\label{theorem:main} Let $G$ be a locally compact group. The following are equivalent. \begin{enumerate}
	\item[$(1)$] $G$ is not amenable.
	\item[$(2)$] There exist a homomorphism $\phi \colon F_{2} \to \mathscr{C}(G)$ and a Borel measurable map $\psi \colon G \to F_{2}$ such that $\psi \circ \phi (g) = \lambda_{g} \circ \psi$ for all $g \in F_{2}$.
	\item[$(3)$] There exist a homomorphism $\phi \colon F_{2} \to \mathscr{B}(G)$ and a Borel measurable map $\psi \colon G \to F_{2}$ such that $\psi \circ \phi (g) = \lambda_{g} \circ \psi$ for all $g \in F_{2}$.
\end{enumerate} \end{thm}

We remark that any map $\phi$ as in~(2) or~(3) of Theorem~\ref{theorem:main} has to be injective. In view of the discussion above, we also note that for finitely generated discrete groups the statement of Theorem~\ref{theorem:main} reduces to Whyte's geometric solution to the von Neumann problem. More specifically, the existence of a map $\psi$ as in~(2) or~(3) above may be thought of as a Borel variant of the semi-regular embedding condition in Corollary~\ref{corollary:whyte.1}. In general, we cannot arrange for $\psi$ to be continuous, as there exist non-amenable connected locally compact groups
		
Both~(2) and~(3) of Theorem~\ref{theorem:main} may be considered relaxed versions of containing $F_{2}$ as a discrete subgroup: according to a result of Feldman and Greenleaf~\cite{FeldmannGreenleaf}, if $H$ is a $\sigma$-compact metrizable closed (e.g., countable discrete) subgroup of a locally compact group $G$, then the right coset projection $G \to H\! \setminus \! G, \, x \mapsto Hx$ admits a Borel measurable cross-section $\tau \colon H \! \setminus \! G \to G$, and hence the $H$-equivariant map \begin{displaymath}
	\psi \colon G \to H, \quad x \mapsto x\tau (Hx)^{-1}
\end{displaymath} is Borel measurable, too. This particularly applies if $H \cong F_{2}$ is discrete.

The proof of Theorem~\ref{theorem:main} combines a result of Rickert resolving the original von Neumann problem for almost connected locally compact groups (Theorem~\ref{theorem:rickert}) with a slight generalization of Whyte's result for coarse spaces (Theorem~\ref{theorem:whyte.2}) and in turn refines an argument of Paterson proving the existence of Borel paradoxical decompositions for non-amenable locally compact groups~\cite{paterson86}. In fact, Theorem~\ref{theorem:main} implies Paterson's result~\cite{paterson86}.

\begin{cor}[Paterson~\cite{paterson86}]\label{corollary:paterson} A locally compact group $G$ is non-amenable if and only if it admits a Borel paradoxical decomposition, i.e., there exist finite partitions $\mathscr{P}$ and $\mathscr{Q}$ of $G$ into Borel subsets and $g_{P},h_{Q} \in G$ $(P \in \mathscr{P}, \, Q \in \mathscr{Q})$ such that \begin{displaymath}
	G = \bigcupdot_{P \in \mathscr{P}} g_{P}P \cupdot \bigcupdot_{Q \in \mathscr{Q}} h_{Q}Q .
\end{displaymath} \end{cor}

This note is organized as follows. Building on some preparatory work concerning coarse spaces done in Section~\ref{section:coarse}, we prove Theorem~\ref{theorem:main} in Section~\ref{section:final}. Since our approach to proving Theorem~\ref{theorem:main} involves wobbling groups, and there has been recent interest in such groups, we furthermore include some complementary remarks about finitely generated subgroups of wobbling groups in Section~\ref{section:more}.

\section{Revisiting Whyte's result}\label{section:coarse}
		
Our proof of Theorem~\ref{theorem:main} will make use of Whyte's argument~\cite{whyte} -- in the form of Corollary~\ref{corollary:whyte.2}. More precisely, we will have to slightly generalize his result from metric spaces to arbitrary coarse spaces. However, this will just require very minor adjustments, and we only include a proof for the sake of completeness.
		
For convenience, let us recall some terminology from coarse geometry as it may be found in~\cite{roe}. For a relation $E \subseteq X \times X$ on a set $X$ and $x \in X$, $A \subseteq X$, let \begin{align*}
	& E[x] \defeq \{ y \in X \mid (x,y) \in E \} , & E[A] \defeq \bigcup \{ E[z] \mid z \in A \} .
\end{align*} A \emph{coarse space} is a pair $(X,\mathscr{E})$ consisting of a set $X$ and a collection $\mathscr{E}$ of subsets of $X \times X$ (called \emph{entourages}) such that \begin{enumerate}
	\item[$\bullet$] the diagonal $\Delta_{X} = \{ (x,x) \mid x \in X \}$ belongs to $\mathscr{E}$,
	\item[$\bullet$] if $F \subseteq E \in \mathscr{E}$, then also $F \in \mathscr{E}$,
	\item[$\bullet$] if $E,F \in \mathscr{E}$, then $E \cup F, E^{-1}, E \circ F \in \mathscr{E}$.
\end{enumerate} A coarse space $(X,\mathscr{E})$ is said to have \emph{bounded geometry} if \begin{displaymath}
	\forall E \in \mathscr{E} \, \forall x \in X \colon \quad E[x] \text{ is finite,}
\end{displaymath} and $(X,\mathscr{E})$ has \emph{uniformly bounded geometry} if \begin{displaymath}
	\forall E \in \mathscr{E} \, \exists m \geq 0 \, \forall x \in X \colon \quad \vert E[x] \vert \leq m .
\end{displaymath}

Among the most important examples of coarse spaces are metric spaces: if $X$ is a metric space, then we obtain a coarse space $(X,\mathscr{E}_{X})$ by setting \begin{displaymath}
	\mathscr{E}_{X} \defeq \{ E \subseteq X \times X \mid \sup \{ d(x,y) \mid (x,y) \in E \} < \infty \} .
\end{displaymath} Another crucial source of examples of coarse spaces is given by group actions. Indeed, if $G$ is a group acting on a set $X$, then we obtain a coarse space $(X,\mathscr{E}_{G})$ of uniformly bounded geometry by \begin{displaymath}
	\mathscr{E}_{G} \defeq \{ R \subseteq X \times X \mid \exists E \subseteq G \text{ finite}\colon \, R \subseteq \{ (x,gx) \mid x \in X, \, g \in E \} \} .
\end{displaymath} Note that the coarse structure induced by a finitely generated group~$G$ acting on itself by left translations coincides with the coarse structure on~$G$ generated by the metric associated with any finite symmetric generated subset of $G$ containing the neutral element. 
		
Now we come to amenability. Adopting the notion from metric coarse geometry, we call a coarse space $(X,\mathscr{E})$ of bounded geometry \emph{amenable} if \begin{displaymath}
	\forall \theta > 1 \, \forall E \in \mathscr{E} \, \exists F \subseteq X \text{ finite, } F \ne \emptyset \colon \quad \vert E[F] \vert \leq \theta \vert F \vert ,
\end{displaymath} which is (easily seen to be) equivalent to saying that \begin{displaymath}
	\exists \theta > 1 \, \forall E \in \mathscr{E} \, \exists F \subseteq X \text{ finite, } F \ne \emptyset \colon \quad \vert E[F] \vert \leq \theta \vert F \vert .
\end{displaymath} This definition is compatible with the existing notion of amenability for group actions (Proposition~\ref{proposition:amenability.of.actions}). Recall that an action of a group $G$ on a set $X$ is \emph{amenable} if the space $\ell^{\infty}(X)$ of all bounded real-valued functions on $X$ admits a $G$-invariant mean, i.e., there exists a positive linear functional $\mu \colon \ell^{\infty}(X) \to \mathbb{R}$ with $\mu (\mathbf{1}) = 1$ and $\mu (f \circ g) = \mu (f)$ for all $f \in \ell^{\infty}(X)$ and $g \in G$.

\begin{prop}[cf.~Rosenblatt~\cite{rosenblatt}]\label{proposition:amenability.of.actions} An action of a group $G$ on a set $X$ is amenable if and only if the coarse space $(X,\mathscr{E}_{G})$ is amenable. \end{prop}

\begin{proof} Generalizing F\o lner's work~\cite{folner} on amenable groups, Rosenblatt~\cite{rosenblatt} showed that an action of a group $G$ on a set $X$ is amenable if and only if \begin{displaymath}
	\forall \theta > 1 \, \forall E \subseteq G \text{ finite } \exists F \subseteq X \text{ finite, } F \ne \emptyset \colon \quad \vert EF \vert \leq \theta \vert F \vert ,
\end{displaymath} which is easily seen to be equivalent to the amenability of $(X,\mathscr{E}_{G})$. \end{proof}
		
Let us turn our attention towards Theorem~\ref{theorem:whyte.1}. A straightforward adaptation of Whyte's original argument readily provides us with the following only very slight generalization (Theorem~\ref{theorem:whyte.2}). For a binary relation $E \subseteq X \times X$, we will denote the associated undirected graph by \begin{displaymath}
	\Gamma (E) \defeq (X,\{ \{ x,y \} \mid (x,y) \in E \}) .
\end{displaymath} Furthermore, let $\gr (f) \defeq \{ (x,f(x)) \mid x \in X \}$ for any map $f \colon X \to Y$. Our proof of Theorem~\ref{theorem:whyte.2} will utilize the simple observation that, for a map $f \colon X \to X$, the graph $\Gamma (\gr (f))$ is a forest, i.e., it contains no cycles, if and only if $f$ has no periodic points, which means that $P (f) \defeq \{ x \in X \mid \exists n \geq 1\colon \, f^{n}(x) = x \}$ is empty.
		
\begin{thm}\label{theorem:whyte.2} Let $d \geq 3$. A coarse space $(X,\mathscr{E})$ of bounded geometry is non-amenable if and only if there is $E \in \mathscr{E}$ such that $\Gamma (E)$ is a $d$-regular forest. \end{thm}
		
\begin{proof} ($\Longleftarrow$) Due to a very standard fact about isoperimetric constants for regular trees~\cite[Example 47]{ParadoxicalDecompositions}, if $E \subseteq X \times X$ is symmetric and $\Gamma (E)$ is a $d$-regular tree, then $\vert E[F]\vert \geq (d-1) \vert F \vert$ for every finite subset $F \subseteq X$. Of course, this property passes to $d$-regular forests, which readily settles the desired implication.
			
($\Longrightarrow$) Suppose that $(X,\mathscr{E})$ is not amenable. Then there is a symmetric entourage $E \in \mathscr{E}$ such that $\vert E[F] \vert \geq d \vert F \vert$ for every finite $F \subseteq X$. Consider the symmetric relation $R \defeq E\setminus \Delta_{X} \subseteq X \times X$. Since $\vert R[x] \vert < \infty$ for every $x \in X$ and \begin{displaymath}
	\vert R[F] \vert \geq \vert E[F]\setminus F \vert \geq \vert E[F] \vert - \vert F \vert \geq (d-1)\vert F \vert
\end{displaymath} for every finite subset $F \subseteq X$, the Hall harem theorem~\cite[Theorem H.4.2]{harem} asserts that there exists a function $f \colon X \to X$ with $\gr (f) \subseteq R$ and $\vert f^{-1}(x) \vert = d-1$ for all $x \in X$. Notice that $f$ does not have any fixed points as $R \cap \Delta_{X} = \emptyset$. Since the set of $f$-orbits of its elements partitions the set $P(f)$, we may choose a subset $P_{0} \subseteq P(f)$ such that $P(f) = \bigcupdot_{x \in P_{0}} \{ f^{n}(x) \mid n \in \mathbb{N} \}$. Furthermore, choose functions $g,h \colon P_{0} \times \mathbb{N} \to X$ such that, for all $x \in P_{0}$ and $n \geq 1$, \begin{enumerate}
	\item[$\bullet$] $g(x,0) = x$ and $h(x,0) = f(x)$,
	\item[$\bullet$] $\{ g(x,n),h(x,n) \} \cap P(f) = \emptyset$,
	\item[$\bullet$] $f(g(x,n)) = g(x,n-1)$ and $f(h(x,n)) = h(x,n-1)$.
\end{enumerate} It follows that $g$ and $h$ are injective functions with disjoint ranges. Now we define $f_{\ast} \colon X \to X$ by setting \begin{displaymath}
	f_{\ast}(x) \defeq \begin{cases}
		g(z,n+2)  & \text{if } x = g(z,n) \text{ for } z \in P_{0} \text{ and even } n \geq 0 , \\
		g(z,n-2)  & \text{if } x = g(z,n) \text{ for } z \in P_{0} \text{ and odd } n \geq 3 , \\
		f^{2}(x)  & \text{if } x = h(z,n) \text{ for } z \in P_{0} \text{ and } n \geq 2 , \\
		f(x) & \text{otherwise}
	\end{cases}
\end{displaymath} for $x \in X$. We observe that \begin{displaymath}
	\gr (f_{\ast}) \subseteq \gr (f^{2})^{-1} \cup \gr (f^{2}) \cup \gr (f) .
\end{displaymath} In particular, $\gr (f_{\ast}) \subseteq E \circ E$ and therefore $\gr (f_{\ast}) \in \mathscr{E}$. Moreover, it follows that $P(f_{\ast}) \subseteq P(f)$. However, for every $x \in P(f)$, there exists a smallest $m \in \mathbb{N}$ such that $f^{m}(x) \in P_{0}$, and we conclude that $f_{\ast}^{m+1}(x) = f_{\ast}(f^{m}(x)) = g(f^{m}(x),2) \notin P(f)$ and hence $f^{m+1}_{\ast}(x) \notin P(f_{\ast})$, which readily implies that $x \notin P(f_{\ast})$. Thus, $P(f_{\ast}) = \emptyset$. In particular, $f_{\ast}$ has no fixed points. Furthermore, \begin{displaymath}
	f_{\ast}^{-1}(x) = \begin{cases}
		(f^{-1}(x) \cup \{ g(z,n-2) \} )\setminus \{ g(z,n+1) \}  & \text{if } x = g(z,n) \text{ for } z \in P_{0} \\ & \text{and even } n \geq 2 , \\
		(f^{-1}(x) \cup \{ g(z,n+2) \} )\setminus \{ g(z,n+1) \}  & \text{if } x = g(z,n) \text{ for } z \in P_{0} \\ & \text{and odd } n \geq 1 , \\
		(f^{-1}(x) \cup \{ h(z,n+2) \} )\setminus \{ h(z,n+1) \}  & \text{if } x = h(z,n) \text{ for } z \in P_{0} \\ & \text{and } n \geq 1 , \\
		(f^{-1}(x) \cup \{ h(z,2) \} )\setminus \{ z \}  & \text{if } x = f(z) \text{ for } z \in P_{0} , \\
		f^{-1}(x) & \text{otherwise}
\end{cases}
\end{displaymath} and thus $\vert f_{\ast}^{-1}(x) \vert = d-1$ for each $x \in X$. Hence, $\Gamma (\gr (f_{\ast}))$ is a $d$-regular forest. \end{proof}
		
Just as Theorem~\ref{theorem:whyte.1} corresponds to Corollary~\ref{corollary:whyte.1}, we can translate Theorem~\ref{theorem:whyte.2} into an equivalent statement about wobbling groups. Given a coarse space $(X,\mathscr{E})$, we define its \emph{wobbling group} (or \emph{group of bounded displacement}) as \begin{displaymath}
	\mathscr{W}(X,\mathscr{E}) \defeq \{ \alpha \in \Sym (X) \mid \gr (\alpha) \in \mathscr{E} \} .
\end{displaymath} Since the $4$-regular tree is isomorphic to the standard Cayley graph of the free group on two generators, we now obtain the following consequence of Theorem~\ref{theorem:whyte.2}.
		
\begin{cor}\label{corollary:whyte.2} A coarse space $X$ of bounded geometry is non-amenable if and only if $F_{2}$ is isomorphic to a semi-regular subgroup of $\mathscr{W}(X)$. \end{cor}

We note that Corollary~\ref{corollary:whyte.2} for group actions has been applied already (though without proof) in the recent work of the author and Thom~\cite[Corollary~5.12]{SchneiderThom}, where a topological version of Whyte's result for general (i.e., not necessarily locally compact) topological groups in terms of perturbed translations is established. In the present note, Corollary~\ref{corollary:whyte.2} will be used to prove Theorem~\ref{theorem:main}, which generalizes Whyte's result to locally compact groups by means of clopen and Borel piecewise translations and is in turn quite different to~\cite[Corollary~5.12]{SchneiderThom}.
		
\section{Proving the main result}\label{section:final}
		
In this section we prove Theorem~\ref{theorem:main}. For the sake of clarity, recall that a locally compact group $G$ is said to be \emph{amenable} if there is a $G$-invariant\footnote{In case of ambiguity, invariance shall always mean \emph{left} invariance.} mean on the space $C_{b}(G)$ of bounded continuous real-valued functions on $G$, i.e., a positive linear map $\mu \colon C_{b}(G) \to \mathbb{R}$ with $\mu (\mathbf{1}) = 1$ and $\mu (f \circ \lambda_{g}) = \mu (f)$ for all $f \in C_{b}(G)$ and $g \in G$. In preparation of the proof of Theorem~\ref{theorem:main}, we note the following standard fact, whose straightforward proof we omit.
		
\begin{lem}\label{lemma:paterson} Let $H$ be a subgroup of a locally compact group $G$ and consider the usual action of $G$ on the set $G/H$ of left cosets of $H$ in $G$. If $\mu \colon \ell^{\infty}(G/H) \to \mathbb{R}$ is a $G$-invariant mean and $\nu \colon C_{b}(H) \to \mathbb{R}$ is an $H$-invariant mean, then a $G$-invariant mean $\xi \colon C_{b}(G) \to \mathbb{R}$ is given by \begin{displaymath}
	\xi (f) \defeq \mu (xH \mapsto \nu ((f \circ \lambda_{x})\vert_{H})) \qquad (f \in C_{b}(G)) .
\end{displaymath} \end{lem}
		
It is a well-known fact (see Section 2 in~\cite{greenleaf}) that a locally compact group~$G$ (considered together with a left Haar measure) is amenable if and only if there exists a $G$-invariant mean on $L^{\infty}(G)$, i.e., a positive linear map $\mu \colon L^{\infty}(G) \to \mathbb{R}$ such that $\mu (\mathbf{1}) = 1$ and $\mu (f \circ \lambda_{g}) = \mu (f)$ for all $f \in C_{b}(G)$ and $g \in G$. An easy calculation now provides us with the following.
		
\begin{lem}\label{lemma:amenability} Let $G$ be a locally compact group. \begin{enumerate}
	\item[$(1)$] A mean $\mu \colon L^{\infty}(G) \to \mathbb{R}$ is $G$-invariant if and only if $\mu$ is $\mathscr{B}(G)$-invariant.
	\item[$(2)$] Let $H$ be a locally compact group, let $\phi \colon H \to \mathscr{B}(G)$ be a homomorphism and $\psi \colon G \to H$ be Borel measurable with $\psi \circ \phi (g) = \lambda_{g} \circ \psi$ for all $g \in H$. If $G$ is amenable, then so is $H$.
\end{enumerate} \end{lem}
			
\begin{proof} (1) Clearly, $\mathscr{B}(G)$-invariance implies $G$-invariance. To prove the converse, suppose that $\mu$ is $G$-invariant. Let $\alpha \in \mathscr{B}(G)$ and let $\mathscr{P}$ be a finite partition of $G$ into Borel subsets and $g_{P} \in G$ ($P \in \mathscr{P}$) with $\alpha\vert_{P} = \lambda_{g_{P}}\vert_{P}$ for each $P \in \mathscr{P}$. Now, \begin{align*}
	\mu (f \circ \alpha) &= \sum_{P \in \mathscr{P}} \mu \left( (f \circ \alpha) \cdot \mathbf{1}_{P} \right) = \sum_{P \in \mathscr{P}} \mu \left( (f \circ \lambda_{g_{P}}) \cdot \mathbf{1}_{P} \right) \\
		&= \sum_{P \in \mathscr{P}} \mu \left( f \cdot \left(\mathbf{1}_{P} \circ \lambda_{g_{P}^{-1}}\right)\right) = \sum_{P \in \mathscr{P}} \mu \left( f \cdot (\mathbf{1}_{g_{P} P})\right) \\
		&= \sum_{P \in \mathscr{P}} \mu \left( f \cdot \mathbf{1}_{\alpha(P)} \right) = \mu (f) 
\end{align*} for every $f \in L^{\infty}(G)$, as desired.
				
(2) Let $\nu \colon L^{\infty}(G) \to \mathbb{R}$ be a $G$-invariant mean. Define $\mu \colon C_{b}(H) \to \mathbb{R}$ by \begin{displaymath}
	\mu (f) \defeq \nu (f \circ \psi)  \qquad (f \in C_{b}(H)) .
\end{displaymath} It is easy to see that $\mu$ is a mean. Furthermore, (1) asserts that \begin{displaymath}
	\mu (f \circ \lambda_{g}) = \nu (f \circ \lambda_{g} \circ \psi) = \nu (f \circ \psi \circ \phi (g)) = \nu (f \circ \psi) = \mu (f)
\end{displaymath} for all $f \in C_{b}(H)$ and $g \in H$. Hence, $\mu$ is $H$-invariant. \end{proof}
			
We note that Lemma~\ref{lemma:amenability} readily settles the implication (3)$\Longrightarrow$(1) of Theorem~\ref{theorem:main}. The remaining part of the proof of Theorem~\ref{theorem:main} will rely on some structure theory for locally compact groups -- most importantly the following remarkable result of Rickert~\cite{rickert2} building on~\cite{rickert1}. We recall that a locally compact group $G$ is said to be \emph{almost connected} if the quotient of $G$ by the connected component of its neutral element is compact.
			
\begin{thm}[Theorem 5.5 in~\cite{rickert2}]\label{theorem:rickert} Any almost connected, non-amenable, locally compact group has a discrete subgroup being isomorphic to $F_{2}$. \end{thm}
			
Now everything is prepared to prove our main result.
			
\begin{proof}[Proof of Theorem~\ref{theorem:main}] Evidently, (2) implies (3) as $\mathscr{C}(G)$ is a subgroup of $\mathscr{B}(G)$. Furthermore, (3) implies (1) due to Lemma~\ref{lemma:amenability} and the non-amenability of $F_{2}$.
				
(1)$\Longrightarrow$(2). Let $G$ be a non-amenable locally compact group. It follows by classical work of van Dantzig~\cite{VanDantzig} that any locally compact group contains an almost connected, open subgroup (see, e.g.,~\cite[Proposition~12.2.2~(c)]{palmer}). Choose any almost connected, open (and hence closed) subgroup $H$ of $G$. We will distinguish two cases depending upon whether $H$ is amenable.
				
\textit{$H$ is not amenable.} According to Theorem~\ref{theorem:rickert}, $H$ contains a discrete subgroup $F$ being isomorphic to $F_{2}$, and so does $G$. By a result of Feldman and Greenleaf~\cite{FeldmannGreenleaf}, the right coset projection $\pi \colon G \to F\! \setminus \! G, \, x \mapsto Fx$ admits a Borel measurable cross-section, i.e., there exists a Borel measurable map $\tau \colon F\! \setminus \! G \to G$ such that $\pi \circ \tau = \id_{F\setminus G}$. Clearly, the $F$-equivariant map $\psi \colon G \to F, \, x \mapsto x\tau(Fx)^{-1}$ is Borel measurable. This readily settles the first case: the maps \begin{displaymath}
	\phi \colon F_{2} \cong F \to \mathscr{C}(G), \quad g \mapsto \lambda_{g}
\end{displaymath} and $\psi$ are as desired.
				
\textit{$H$ is amenable.} Since $G$ is not amenable, Lemma~\ref{lemma:paterson} implies that the action of $G$ on the set $G/H$ is not amenable. By Proposition~\ref{proposition:amenability.of.actions}, this means that the coarse space $X \defeq (G/H,\mathscr{E}_{G})$ is not amenable. Due to Corollary~\ref{corollary:whyte.2}, there exists an embedding $\phi \colon F_{2} = F(a,b) \to \mathscr{W}(X)$ such that $\phi (F_{2})$ is semi-regular. Thus, by definition of~$\mathscr{W}(X)$, there exists some finite subset $E \subseteq G$ such that \begin{displaymath}
		\forall x \in \{ a,b \} \, \forall z \in X \, \exists g \in E \colon \quad \phi (x)(z) = gz .
\end{displaymath} Hence, we find a finite partition $\mathscr{P}$ of $X$ along with $g_{P}, h_{P} \in E$ ($P \in \mathscr{P}$) such that $\phi (a)\vert_{P} = \lambda_{g_{P}}\vert_{P}$ and $\phi (b)\vert_{P} = \lambda_{h_{P}}\vert_{P}$ for every $P \in \mathscr{P}$. Consider the projection $\pi \colon G \to G/H, \, x \mapsto xH$. Since $H$ is an open subgroup of $G$, the quotient topology on $G/H$, i.e., the topology induced by $\pi$, is discrete. So, $\pi^{-1}(\mathscr{P}) = \{ \pi^{-1}(P) \mid P \in \mathscr{P} \}$ is a finite partition of $G$ into clopen subsets. What is more, \begin{align*}
	G &= \bigcupdot_{P \in \mathscr{P}} \pi^{-1}(\phi(a)(P)) = \bigcupdot_{P \in \mathscr{P}} \pi^{-1}(g_{P}P) = \bigcupdot_{P \in \mathscr{P}} g_{P}\pi^{-1}(P) ,\\
	G &= \bigcupdot_{P \in \mathscr{P}} \pi^{-1}(\phi(b)(P)) = \bigcupdot_{P \in \mathscr{P}} \pi^{-1}(h_{P}P) = \bigcupdot_{P \in \mathscr{P}} h_{P}\pi^{-1}(P) .
\end{align*} Therefore, we may define $\overline{\phi} \colon \{ a,b \} \to \mathscr{C}(G)$ by setting \begin{align*}
	&\overline{\phi}(a)\vert_{\pi^{-1}(P)} = \lambda_{g_{P}} \vert_{\pi^{-1}(P)} , \qquad \overline{\phi}(b)\vert_{\pi^{-1}(P)} = \lambda_{h_{P}} \vert_{\pi^{-1}(P)} & (P \in \mathscr{P}) .
\end{align*} Consider the unique homomorphism $\phi^{\ast}\colon F_{2} \to \mathscr{C}(G)$ satisfying $\phi^{\ast}\vert_{\{ a,b \}} = \overline{\phi}$. Since $\pi \circ \overline{\phi}(x) = \phi (x) \circ \pi$ for each $x \in \{ a,b \}$, it follows that $\pi \circ \phi^{\ast}(w) = \phi (w) \circ \pi$ for every $w \in F_{2}$. Appealing to Remark~\ref{remark:semiregular}, we find a mapping $\psi \colon G/H \to F_{2}$ such that $\psi (\phi (w)(z)) = w\psi (z)$ for all $w \in F_{2}$ and $z \in G/H$. Since the quotient space $G/H$ is discrete, the map $\psi^{\ast} \defeq \psi \circ \pi \colon G \to F_{2}$ is continuous and therefore Borel measurable. Finally, we note that \begin{displaymath}
	\psi^{\ast}(\phi^{\ast}(w)(x)) = \psi (\pi (\phi^{\ast}(w)(x))) = \psi (\phi (w)(\pi (x))) = w \psi (\pi (x)) = w \psi^{\ast}(x)
\end{displaymath} for all $w \in F_{2}$ and $x \in G$, as desired. This completes the proof. \end{proof} 
			
Let us deduce Paterson's result~\cite{paterson86} from Theorem~\ref{theorem:main}.
			
\begin{proof}[Proof of Corollary~\ref{corollary:paterson}] ($\Longleftarrow$) This is clear.
				
($\Longrightarrow$) Let $G$ be a non-amenable locally compact group. By Theorem~\ref{theorem:main}, there exist a homomorphism $\phi \colon F_{2} \to \mathscr{B}(G)$ and a Borel measurable map $\psi \colon G \to F_{2}$ with $\psi \circ \phi (g) = \lambda_{g} \circ \psi$ for all $g \in F_{2}$. Consider any paradoxical decomposition of $F_{2}$ given by $\mathscr{P}$, $\mathscr{Q}$, $(g_{P})_{P \in \mathscr{P}}$, $(h_{Q})_{Q \in \mathscr{Q}}$. Taking a common refinement of suitable finite Borel partitions of $G$ corresponding to the elements $\phi (g_{P}), \phi (h_{Q}) \in \mathscr{B}(G)$ ($P \in \mathscr{P}, \, Q \in \mathscr{Q}$), we obtain a finite Borel partition $\mathscr{R}$ of $G$ along with mappings $\sigma \colon \mathscr{P} \times \mathscr{R} \to G$ and $\tau \colon \mathscr{Q} \times \mathscr{R} \to G$ such that \begin{align*}
	&\phi (g_{P})\vert_{R} = \lambda_{\sigma (P,R)}\vert_{R} &\phi (h_{Q})\vert_{R} = \lambda_{\tau (Q,R)}\vert_{R}
\end{align*} for all $P \in \mathscr{P}$, $Q \in \mathscr{Q}$, and $R \in \mathscr{R}$. By $\psi$ being Borel measurable, the refinements $\psi^{-1}(\mathscr{P}) \vee \mathscr{R}$ and $\psi^{-1}(\mathscr{Q}) \vee \mathscr{R}$ are finite Borel partitions of $G$. What is more, \begin{align*}
	G \, &= \bigcupdot_{P \in \mathscr{P}} \psi^{-1}(g_{P}P) \cupdot \bigcupdot_{Q \in \mathscr{Q}} \psi^{-1}(h_{Q}Q) \\
		&= \bigcupdot_{P \in \mathscr{P}} \phi(g_{P})(\psi^{-1}(P)) \cupdot \bigcupdot_{Q \in \mathscr{Q}} \phi (h_{Q})(\psi^{-1}(Q)) \\
		&= \bigcupdot_{(P,R) \in \mathscr{P} \times \mathscr{R}} \phi(g_{P})(\psi^{-1}(P) \cap R) \ \cupdot \bigcupdot_{(Q,R) \in \mathscr{Q} \times \mathscr{R}} \phi (h_{Q})(\psi^{-1}(Q) \cap R) \\
		&= \bigcupdot_{(P,R) \in \mathscr{P} \times \mathscr{R}} \sigma (P,R)(\psi^{-1}(P) \cap R) \ \cupdot \bigcupdot_{(Q,R) \in \mathscr{Q} \times \mathscr{R}} \tau (Q,R)(\psi^{-1}(Q) \cap R) .
\end{align*} Thus, the data \begin{displaymath}
	\psi^{-1}(\mathscr{P}) \vee \mathscr{R} , \quad \psi^{-1}(\mathscr{Q}) \vee \mathscr{R} , \quad (\sigma (P,R))_{(P,R) \in \mathscr{P} \times \mathscr{R}} , \quad (\tau (Q,R))_{(Q,R) \in \mathscr{Q} \times \mathscr{R}}
\end{displaymath} constitute a Borel paradoxical decomposition of $G$. \end{proof}

\section{Further remarks on wobbling groups}\label{section:more}

We are going to conclude with some additional remarks about wobbling groups, which we consider noteworthy complements of Corollary~\ref{corollary:whyte.2}. As van~Douwen's result~\cite{VanDouwen} shows, the presence of $F_{2}$ as a subgroup of the wobbling group does not imply the non-amenability of a coarse space. As it turns out, containment of $F_{2}$ is just a witness for positive asymptotic dimension (Proposition~\ref{proposition:dimension}).

Let us once again recall some terminology from~\cite{roe}. The \emph{asymptotic dimension} $\asdim (X,\mathscr{E})$ of a coarse space $(X,\mathscr{E})$ is defined as the infimum of all those $n \in \mathbb{N}$ such that, for every $E \in \mathscr{E}$, there exist $\mathscr{C}_{0},\ldots,\mathscr{C}_{n} \subseteq \mathscr{P}(X)$ with \begin{enumerate}
	\item[$\bullet$] $X = \bigcup \mathscr{C}_{0} \cup \ldots \cup \bigcup \mathscr{C}_{n}$,
	\item[$\bullet$] $(C \times D) \cap E = \emptyset$ for all $i \in \{ 0,\ldots ,n \}$ and $C,D \in \mathscr{C}_{i}$ with $C \ne D$,
	\item[$\bullet$] $\bigcup \{ C \times C \mid C \in \mathscr{C}_{i}, \, i \in \{ 0,\ldots,n \} \} \in \mathscr{E}$.
\end{enumerate} The concept of asymptotic dimension was first introduced for metric spaces by Gromov~\cite{GromovAsymptotic} and later extended to coarse spaces by Roe~\cite{roe}. We refer to~\cite{roe} for a thorough discussion of asymptotic dimension, related results and examples.

As we aim to describe positive asymptotic dimension in algebraic terms, we will unravel the zero-dimensional case in the following lemma. Let us denote by $[R]$ the equivalence relation on a set $X$ generated by a given binary relation $R \subseteq X \times X$.

\begin{lem}\label{lemma:vanishing.asymptotic.dimension} Let $(X,\mathscr{E})$ be a coarse space. Then $\asdim (X,\mathscr{E}) = 0$ if and only if $[E] \in \mathscr{E}$ for every $E \in \mathscr{E}$. \end{lem}

\begin{proof}($\Longrightarrow$) Let $E \in \mathscr{E}$. Without loss of generality, assume that $E$ contains $\Delta_{X}$. As $\asdim (X,\mathscr{E}) = 0$, there exists $\mathscr{C}_{0} \subseteq \mathscr{P}(X)$ such that \begin{enumerate}
	\item[$(1)$] $X = \bigcup \mathscr{C}_{0}$,
	\item[$(2)$] $(C \times D) \cap E = \emptyset$ for all $C,D \in \mathscr{C}_{0}$ with $C \ne D$,
	\item[$(3)$] $\bigcup \{ C \times C \mid C \in \mathscr{C}_{0} \} \in \mathscr{E}$.
\end{enumerate} As $\Delta_{X} \subseteq E$, assertion~(2) implies that any two distinct members of $\mathscr{C}_{0}$ are disjoint. Hence, (1) gives that $\mathscr{C}_{0}$ is a partition of $X$. By~(2), the induced equivalence relation $R \defeq \bigcup \{ C \times C \mid C \in \mathscr{C}_{0} \}$ contains $E$, thus $[E]$. By~(3), it follows that $[E] \in \mathscr{E}$.

($\Longleftarrow$) Let $E \in \mathscr{E}$. It is straightforward to check that $\mathscr{C}_{0} \defeq \{ [E][x] \mid x \in X \}$ has the desired properties. Hence, $\asdim (X,\mathscr{E}) = 0$. \end{proof}

Our proof of Proposition~\ref{proposition:dimension} below will rely upon the following slight modification of the standard argument for residual finiteness of free groups. For an element $w \in F_{2} = F(a,b)$, let us denote by $\vert w \vert$ the \emph{length} of $w$ with respect to the generators $a$ and $b$, i.e., the smallest integer $n \geq 0$ such that $w$ can be represented as a word of length $n$ in the letters $a,a^{-1},b,b^{-1}$.

\begin{lem}\label{lemma:free.group} Let $w \in F_{2}$ with $w \ne e$ and let $M \defeq \{ 0,\ldots ,2\vert w \vert \}$. Then there exists a homomorphism $\phi \colon F_{2} \to \Sym (M)$ such that $\phi (w) \ne e$ and $\vert \phi (v)(i) - i \vert \leq 2 \vert v \vert$ for all $i \in M$ and $v \in F_{2}$. \end{lem}

\begin{proof} Let $(k_{0},\ldots,k_{n}) \in (\mathbb{Z}\setminus \{ 0\} )^{n} \times \mathbb{Z}$ and $(\ell_{0},\ldots,\ell_{n}) \in \mathbb{Z} \times (\mathbb{Z}\setminus \{ 0\} )^{n}$ such that $w = a^{k_{n}}b^{\ell_{n}}\cdots a^{k_{0}}b^{\ell_{0}}$. Of course, $\vert w \vert = \sum_{i = 0}^{n} \vert k_{i} \vert + \sum_{i = 0}^{n} \vert \ell_{i} \vert$. Let \begin{align*}
		&\alpha_{i} \defeq \sum\nolimits_{j=0}^{i-1} \vert k_{j} \vert + \sum\nolimits_{j=0}^{i} \vert \ell_{j} \vert , &\beta_{i} \defeq \sum\nolimits_{j=0}^{i-1} \vert k_{j} \vert + \sum\nolimits_{j=0}^{i-1} \vert \ell_{j} \vert
\end{align*} for $i \in \{ 0,\ldots,n\}$ and let $\beta_{n+1} \defeq \vert w \vert$. We will define a map $\phi \colon \{ a,b \} \to \Sym (M)$. First, let us define $\phi (a) \in \Sym (M)$ by case analysis as follows: if $i \in [2\alpha_{j},2\beta_{j+1}]$ for some $j \in \{ 0,\ldots,n\}$ with $k_{j} > 0$, then \begin{displaymath}
	\qquad \phi(a) (i) \defeq \begin{cases}
	i+2 & \text{if $i$ is even and } i \in [2\alpha_{j},2\beta_{j+1}-2], \\
	i-1 & \text{if } i = 2\beta_{j+1} , \\
	i-2 & \text{if $i$ is odd and } i \in [2\alpha_{j}+3,2\beta_{j+1}-1],  \\
	i-1 & \text{if } i = 2\alpha_{j} + 1 ,
	\end{cases}
\end{displaymath} if $i \in [2\alpha_{j},2\beta_{j+1}]$ for some $j \in \{ 0,\ldots,n\}$ with $k_{j} < 0$, then \begin{displaymath}
	\qquad \phi(a) (i) \defeq \begin{cases}
	i-2 & \text{if $i$ is even and } i \in [2\alpha_{j}+2,2\beta_{j+1}], \\
	i+1 & \text{if } i = 2\alpha_{j} , \\
	i+2 & \text{if $i$ is odd and } i \in [2\alpha_{j}+1,2\beta_{j+1}-3], \\
	i+1 & \text{if } i = 2\beta_{j+1} - 1 ,
	\end{cases}
\end{displaymath} and if $i \notin \bigcup \{ [2\alpha_{j},2\beta_{j+1}] \mid j \in \{ 0,\ldots,n\}, \, k_{j} \ne 0 \}$, then $\phi (a)(i) \defeq i$. Analogously, let us define $\phi (b) \in \Sym (M)$ by case analysis as follows: if $i \in [2\beta_{j},2\alpha_{j}]$ for some $j \in \{ 0,\ldots,n \}$ with $\ell_{j} > 0$, then \begin{displaymath}
	\qquad \phi(b) (i) \defeq \begin{cases}
	i+2 & \text{if $i$ is even and } i \in [2\beta_{j},2\alpha_{j}-2],  \\
	i-1 & \text{if } i = 2\alpha_{j} , \\
	i-2 & \text{if $i$ is odd and } i \in [2\beta_{j}+3,2\alpha_{j}-1], \\
	i-1 & \text{if } i = 2\beta_{j} + 1 ,
	\end{cases}
\end{displaymath} if $i \in [2\beta_{j},2\alpha_{j}]$ for some $j \in \{ 0,\ldots,n \}$ with $\ell_{j} < 0$, then \begin{displaymath}
	\qquad \phi(b) (i) \defeq \begin{cases}
	i-2 & \text{if $i$ is even and } i \in [2\beta_{j}+2,2\alpha_{j}],  \\
	i+1 & \text{if } i = 2\beta_{j} , \\
	i+2 & \text{if $i$ is odd and } i \in [2\beta_{j}+1,2\alpha_{j}-3], \\
	i+1 & \text{if } i = 2\alpha_{j} - 1 ,
	\end{cases}
\end{displaymath} and if $i \notin \bigcup \{ [2\beta_{j},2\alpha_{j}] \mid j \in \{ 0,\ldots, n\}, \, \ell_{j} \ne 0 \}$, then $\phi (b)(i) \defeq i$. It is easy to check that $\phi(a)$ and $\phi(b)$ are well-defined permutations of $M$, and that moreover $\vert \phi (x)(i) - i \vert \leq 2$ for each $x \in \{ a,b \}$ and all $i \in M$. Considering the unique homomorphism $\phi^{\ast} \colon F_{2} \to \Sym (M)$ with $\phi^{\ast}\vert_{\{ a,b \}} = \phi$, we observe that \begin{displaymath}
	\phi^{\ast}(w)(0) = \left(\phi(a)^{k_{n}}\phi(b)^{\ell_{n}} \cdots \phi(a)^{k_{0}}\phi(b)^{\ell_{0}}\right)(0) = 2\vert w \vert 
\end{displaymath} and thus $\phi^{\ast}(w) \ne e$. Also, $\vert \phi^{\ast} (v)(i) - i \vert \leq 2 \vert v \vert$ for all $i \in M$ and $v \in F_{2}$. \end{proof}

For the sake of clarity, we recall that a group is \emph{locally finite} if each of its finitely generated subgroups is finite. For a subset $S$ of a group $G$, we will denote by $\langle S \rangle$ the subgroup of $G$ generated by $S$.

\begin{prop}\label{proposition:dimension} Let $X$ be a coarse space of uniformly bounded geometry. The following are equivalent. \begin{enumerate}
	\item[$(1)$] $\asdim (X) > 0$.
	\item[$(2)$] $\mathscr{W}(X)$ is not locally finite.
	\item[$(3)$] $F_{2}$ embeds into $\mathscr{W}(X)$.
\end{enumerate} \end{prop}

\begin{proof} We will denote by $\mathscr{E}$ the coarse structure of $X$.

(2)$\Longrightarrow$(1). Let us recall a general fact: for a finite group $G$ and any set $M$, the group $G^{M}$ is locally finite. Indeed, considering a finite subset $S \subseteq G^{M}$ and the induced equivalence relation $R \defeq \{ (x,y) \in M \times M \mid \forall \alpha \in S \colon \, \alpha (x) = \alpha (y) \}$ on $M$, we observe that $N \defeq \{ R[x] \mid x \in M \}$ is finite, due to $G$ and $S$ being finite. The map $\pi \colon M \to N , \, x \mapsto R[x]$ induces a homomorphism $\phi \colon G^{N} \to G^{M}, \, \alpha \mapsto \alpha \circ \pi$. Evidently, $S$ is contained in the finite group $\phi (G^{N})$, and so is $\langle S \rangle$.

Suppose now that $\asdim (X) = 0$. Consider a finite subset $S \subseteq \mathscr{W}(X)$. We aim to show that $H \defeq \langle S \rangle$ is finite. To this end, we first observe that \begin{displaymath}
	\phi \colon H \to \prod\nolimits_{x \in X} \Sym (Hx) , \quad \alpha \mapsto (\alpha|_{Hx})_{x \in X}
\end{displaymath} constitutes a well-defined embedding. Since $D \defeq \bigcup \{ \gr (\alpha) \mid \alpha \in S \}$ belongs to $\mathscr{E}$, Lemma~\ref{lemma:vanishing.asymptotic.dimension} asserts that $E \defeq [D] \in \mathscr{E}$, too. Note that $\gr (\alpha) \in E$ for all $\alpha \in H$. Hence, $Hx \subseteq E[x]$ for every $x \in X$. Due to $X$ having uniformly bounded geometry, there exists $m \geq 0$ such that $\vert E[x] \vert \leq m$ and thus $\vert Hx \vert \leq m$ for every $x \in X$. Now, let $M \defeq \{ 0,\ldots,m-1 \}$. It follows that the group $\prod_{x \in X} \Sym (Hx)$ is isomorphic to a subgroup of $\Sym (M)^{X}$, and so is $H$ by virtue of $\phi$. Since $H$ is finitely generated and $\Sym (M)^{X}$ is locally finite by the remark above, this implies that $H$ is finite.

(3)$\Longrightarrow$(2). This is trivial.

(1)$\Longrightarrow$(3). Suppose that $\asdim (X) > 0$. By Lemma~\ref{lemma:vanishing.asymptotic.dimension}, there exists $E \in \mathscr{E}$ such that $[E] \notin \mathscr{E}$. Without loss of generality, we may assume that $\Delta_{X} \subseteq E = E^{-1}$. Hence, $[E] = \bigcup \{ E^{n} \mid n \in \mathbb{N} \}$. For each $n \in \mathbb{N}$, let us define \begin{displaymath}
	\left. T_{n} \defeq \left\{ x \in X^{n+1} \, \right| \vert \{ x_{0},\ldots,x_{n} \} \vert = n+1 , \, \forall i \in \{ 0,\ldots ,n-1 \} \colon \, (x_{i},x_{i+1}) \in E \right\} .
\end{displaymath} \textit{Claim.} For every $n \in \mathbb{N}$ and every finite subset $F \subseteq X$, there exists $x \in T_{n}$ such that $\{ x_{0},\ldots,x_{n} \} \cap F = \emptyset$.

\textit{Proof of claim.} Let $n \in \mathbb{N}$ and let $F \subseteq X$ be finite. Put $\ell \defeq (n+1)(\vert F \vert +1)$. Since $E \in \mathscr{E}$ and $[E] \notin \mathscr{E}$, we conclude that $E^{\ell} \nsubseteq E^{\ell-1}$. Let $x_{0},\ldots ,x_{\ell} \in X$ such that $(x_{0},x_{\ell}) \notin E^{\ell -1}$ and $(x_{i},x_{i+1}) \in E$ for every $i \in \{ 0,\ldots,\ell -1 \}$. As $\Delta_{X} \subseteq E$, it follows that $\vert \{ x_{0},\ldots ,x_{\ell} \} \vert = \ell + 1$. Applying the pigeonhole principle, we find some $j \in \{ 0,\ldots ,\ell -n \}$ such that $\{ x_{j},\ldots ,x_{j+ n} \} \cap F = \emptyset$. Hence, $y_{0}\defeq x_{j}, \ldots , y_{n} \defeq x_{j+n}$ are as desired. \qed

Since $N \defeq F_{2}\setminus \{ e \}$ is countable, we may recursively apply the claim above and choose a family $(x_{w})_{w \in N}$ such that \begin{enumerate}
	\item[(i)] $x_{w} \in T_{2 \vert w \vert}$ for every $w \in N$,
	\item[(ii)] $\{ x_{w,0} ,\ldots, x_{w,2\vert w \vert} \} \cap \{ x_{v,0} ,\ldots, x_{v,2\vert v \vert} \} = \emptyset$ for any two distinct $v,w \in N$.
\end{enumerate} Let $w \in N$ and define $D_{w} \defeq \{ x_{w,0} ,\ldots, x_{w,2\vert w \vert} \}$. Due to Lemma~\ref{lemma:free.group}, there exists a homomorphism $\phi_{w} \colon F_{2} \to \Sym (D_{w})$ such that $\phi_{w}(w) \ne e$ and \begin{displaymath}
	\phi_{w}(v)(x_{w,i}) \in \{ x_{w,j} \mid  j \in \{ 0 ,\ldots , 2\vert w \vert \} , \, \vert i-j \vert \leq 2\vert v \vert \}
\end{displaymath} for all $v \in F_{2}$, $i \in \{ 0,\ldots,2\vert w \vert\}$. Since $(x_{w,i},x_{w,i+1}) \in E$ for $i \in \{ 0,\ldots,2\vert w \vert -1 \}$, it follows that $\gr (\phi_{w}(v)) \subseteq E^{2 \vert v \vert}$ for all $v \in F_{2}$. As $D_{w}$ and $D_{v}$ are disjoint for any distinct $v,w \in N$, we may define a homomorphism $\phi \colon F_{2} \to \Sym (X)$ by setting \begin{displaymath}
	\phi (v)(x) \defeq \begin{cases}
		\phi_{w}(v)(x) & \text{if } x \in D_{w} \text{ for some } w \in N , \\
		x & \text{otherwise}
	\end{cases}
\end{displaymath} for $v \in F_{2}$ and $x \in X$. By construction, $\phi$ is an embedding, and furthermore \begin{displaymath}
	\gr (\phi (v)) \, \subseteq \, \Delta_{X} \cup \bigcup \{ \gr (\phi_{w}(v)) \mid w \in N \} \, \subseteq \,  E^{2\vert v \vert} \, \in \, \mathscr{E}
\end{displaymath} for every $v \in F_{2}$. Hence, the image of $\phi$ is contained in $\mathscr{W}(X)$, as desired. \end{proof}

\begin{remark}\label{remark:unbounded} The assumption of uniformly bounded geometry in Theorem~\ref{theorem:main} is needed only to prove that~(2) implies~(1). In fact, a similar argument as in the proof of (1)$\Longrightarrow$(3) (not involving Lemma~\ref{lemma:free.group} though) shows that the wobbling group of any coarse space not having uniformly bounded geometry contains an isomorphic copy of $\prod_{n \in \mathbb{N}} \Sym (n)$, hence~$F_{2}$. \end{remark}

One might wonder whether Proposition~\ref{proposition:dimension} could have been deduced readily from van~Douwen's result~\cite{VanDouwen} on $F_{2}$ embedding into $\mathscr{W}(\mathbb{Z})$. However, there exist uniformly discrete metric spaces of uniformly bounded geometry and positive asymptotic dimension whose wobbling group does not contain an isomorphic copy of~$\mathscr{W}(\mathbb{Z})$ (see Example~\ref{example.2}). We clarify the situation in Proposition~\ref{proposition:bornologous}.

As usual, a group is called \emph{residually finite} if it embeds into a product of finite groups, and a group is called \emph{locally residually finite} if each of its finitely generated subgroups is residually finite. Let us recall from~\cite{roe} that a map $f \colon X \to Y$ between two coarse spaces $X$ and $Y$ is \emph{bornologous} if, for every entourage $E$ of $X$, the set $\{ (f(x),f(y)) \mid (x,y) \in E \}$ is an entourage of $Y$.

\begin{prop}\label{proposition:bornologous} Let $X$ be a coarse space. The following are equivalent. \begin{enumerate}
	\item[$(1)$] There is a bornologous injection from $\mathbb{Z}$ into $X$.
	\item[$(2)$] $\mathscr{W}(X)$ is not locally residually finite.
	\item[$(3)$] $\mathscr{W}(X)$ contains a subgroup being isomorphic to $\mathscr{W}(\mathbb{Z})$.
\end{enumerate} \end{prop}

\begin{remark}\label{remark:bornologous.injection} (i) For groups there is no difference between positive asymptotic dimension and the existence of a bornologous injection of $\mathbb{Z}$: a group has asymptotic dimension $0$ if and only if it is locally finite, and any group which is not locally finite admits a bornologous injection of $\mathbb{Z}$ by a standard compactness argument (see, e.g.,~\cite[IV.A.12]{DeLaHarpe}). However, for arbitrary coarse spaces, even of uniformly bounded geometry, the situation is slightly different (see Example~\ref{example.2}).
	
(ii) One may equivalently replace $\mathbb{Z}$ by $\mathbb{N}$ in item~(1) of Proposition~\ref{proposition:bornologous}: on the one hand, the inclusion map constitutes a bornologous injection from $\mathbb{N}$ into $\mathbb{Z}$; on the other hand, there is a bornologous bijection $f \colon \mathbb{Z} \to \mathbb{N}$ given by \begin{displaymath}
	f(n) \defeq \begin{cases}
	2n & \text{if } n \geq 0 , \\
	2\vert n \vert - 1 & \text{if } n<0
	\end{cases} \qquad (n \in \mathbb{Z}) .
\end{displaymath} Unless explicitly stated otherwise, we always understand $\mathbb{N}$ as being equipped with the coarse structure generated by the usual (i.e., Euclidean) metric.

(iii) Any bornologous injection $f \colon X \to Y$ between two coarse spaces $X$ and $Y$ induces an embedding $\phi \colon \mathscr{W}(X) \to \mathscr{W}(Y)$ via \begin{displaymath}
	\phi(\alpha)(y) \defeq \begin{cases}
		f(\alpha (f^{-1}(y))) & \text{if } y \in f(X) , \\
		y & \text{otherwise}
	\end{cases} \qquad (\alpha \in \mathscr{W}(X), \, y \in Y) .
\end{displaymath} Hence, by (ii), the groups $\mathscr{W}(\mathbb{N})$ and $\mathscr{W}(\mathbb{Z})$ mutually embed into each other, and thus $\mathbb{Z}$ may equivalently be replaced by~$\mathbb{N}$ in item~(3) of Proposition~\ref{proposition:bornologous}. \end{remark}

\begin{proof}[Proof of Proposition~\ref{proposition:bornologous}] (1)$\Longrightarrow$(3). This is due to Remark~\ref{remark:bornologous.injection}(iii).

(3)$\Longrightarrow$(2). It suffices to show that $\mathscr{W}(\mathbb{Z})$ is not locally residually finite. A result of Gruenberg~\cite{gruenberg} states that, for a finite group $F$, the restricted wreath product $F \wr \mathbb{Z} = F^{(\mathbb{Z})} \rtimes \mathbb{Z}$ (i.e., the lamplighter group over $F$) is residually finite if and only if $F$ is abelian. For $n \geq 1$, the action of $\Sym (n) \wr \mathbb{Z}$ on $\mathbb{Z} = \bigcupdot_{r=0}^{n-1} n \mathbb{Z} + r$ given by \begin{displaymath}
	(\alpha , m).(nk + r) \defeq n(m+k) + \alpha_{m+k}(r) \quad \left( \alpha \in \Sym (n)^{(\mathbb{Z})}, \, m,k \in \mathbb{Z}, \, 0 \leq r < n \right)
\end{displaymath} defines an embedding of $\Sym (n) \wr \mathbb{Z}$ into $\Sym (\mathbb{Z})$, the image of which is contained in $\mathscr{W}(\mathbb{Z})$ as $\sup_{z \in \mathbb{Z}} \vert z - (\alpha , m).z \vert \leq n (\vert m \vert + 1)$ for every $(\alpha,m) \in\Sym (n) \wr \mathbb{Z}$. Since the embedded lamplighter groups are finitely generated and not residually finite for $n \geq 3$, it follows that $\mathscr{W}(\mathbb{Z})$ is not locally residually finite.

(2)$\Longrightarrow$(1). Let $\mathscr{E}$ denote the coarse structure of $X$. If $X$ does not have bounded geometry, then there exist $E \in \mathscr{E}$ and $x \in X$ such that $E[x]$ is infinite, and any thus existing injection $f \colon \mathbb{Z} \to X$ with $f(\mathbb{Z}) \subseteq E[x]$ is bornologous. Hence, we may without loss of generality assume that $X$ has bounded geometry. On the other hand, there must exist $E \in \mathscr{E}$ and $x \in X$ such that $[E][x]$ is infinite. Otherwise, $\mathscr{W}(X)$ would have to be locally residually finite: for any finite subset $F \subseteq \mathscr{W}(X)$, since $E \defeq \bigcup \{ \gr (\alpha) \mid \alpha \in F \} \in \mathscr{E}$, the homomorphism \begin{displaymath}
	\langle F \rangle \to \prod\nolimits_{x \in X} \Sym ([E][x]) , \quad \alpha \mapsto \left(\alpha\vert_{[E][x]}\right)_{x \in X}
\end{displaymath} would embed $\langle F \rangle$ into a product of finite groups. So, let $E \in \mathscr{E}$ and $x \in X$ such that $[E][x]$ is infinite. Without loss of generality, we may assume that $\Delta_{X} \subseteq E = E^{-1}$. Therefore, $[E] = \bigcup \{ E^{n} \mid n \in \mathbb{N} \}$. We conclude that $E^{n}[x] \ne E^{n+1}[x]$ and thus \begin{displaymath}
	\left. R_{n} \defeq \left\{ f \in X^{\mathbb{N}} \, \right\vert f_{0} = x, \, \vert\{ f_{0}, \ldots,f_{n} \} \vert = n + 1 , \, \forall i \in \mathbb{N} \colon \, (f_{i},f_{i+1}) \in E \right\}
\end{displaymath} is non-empty for all $n \in \mathbb{N}$. As $(R_{n})_{n \in \mathbb{N}}$ is a chain of closed subsets of the compact topological space $\prod\nolimits_{m \in \mathbb{N}} E^{m}[x]$, we have $R \defeq \bigcap_{n \in \mathbb{N}} R_{n} \ne \emptyset$. Since any member of $R$ is a bornologous injection from $\mathbb{N}$ into $X$, this implies~(1) by Remark~\ref{remark:bornologous.injection}(ii). \end{proof}

\begin{exmpl}\label{example.2} Let $\mathscr{I}$ be a partition of $\mathbb{N}$ into finite intervals with $\sup_{I \in \mathscr{I}} \vert I \vert = \infty$. Consider the metric space $X \defeq (\mathbb{N},d)$ given by \begin{displaymath}
	d(x,y) \defeq \begin{cases}
		\vert x - y \vert & \text{if } x,y \in I \text{ for some } I \in \mathscr{I} , \\
		\max (x,y) & \text{otherwise}
	\end{cases} \qquad (x,y \in \mathbb{N}) .
\end{displaymath} It is easy to see that $X$ has uniformly bounded geometry. Moreover, by Lemma~\ref{lemma:vanishing.asymptotic.dimension} and the unboundedness assumption for the interval lengths, it follows that $X$ has positive asymptotic dimension. On the other hand, essentially by finiteness of the considered intervals, there is no bornologous injection from $\mathbb{N}$ into $X$. Due to Proposition~\ref{proposition:bornologous}, this readily implies that $\mathscr{W}(\mathbb{Z})$ does not embed into $\mathscr{W}(X)$. \end{exmpl} 

The interplay between certain geometric properties of coarse spaces on the one hand and algebraic peculiarities of their wobbling groups on the other is a subject of recent attention~\cite{juschenko,cornulier}. It would be interesting to have further results in that direction, e.g., to understand if (and how) specific positive values for the asymptotic dimension may be characterized in terms of wobbling groups.
			
\section*{Acknowledgments}
			
The author would like to thank Andreas Thom for interesting discussions about Whyte's variant of the von Neumann conjecture, as well as Warren Moors and Jens Zumbr\"agel for their helpful comments on earlier versions of this note.

\end{document}